\documentclass[11pt,righttag]{amsart}
\usepackage{diagrams}
\usepackage{graphicx}

\newtheorem{thm}{Theorem}[section]
\newtheorem{lemma}[thm]{Lemma}
\newtheorem{cor}[thm]{Corollary}
\newtheorem{prop}[thm]{Proposition}

\theoremstyle{definition}
\newtheorem{defn}[thm]{Definition}

\newtheorem{rmk}[thm]{Remark}

\newtheorem{example}[thm]{Example}

\newcommand{\F}{{\mathbb F}}

\begin{document}


\title[Weakly Cohen-Macaulay Posets and a Class of Algebras]{Weakly Cohen-Macaulay posets and a class of finite-dimensional graded quadratic algebras}


\author{Tyler Kloefkorn}
\address{Department of Mathematics \\
University of Arizona \\
Tucson, AZ 85721}
\email{tkloefkorn@math.arizona.edu}


\begin{abstract}
\baselineskip12pt
To a finite ranked poset $\Gamma$ we associate a finite-dimensional graded quadratic algebra $R_\Gamma$. Assuming $\Gamma$ satisfies a combinatorial condition known as uniform, $R_{\Gamma}$ is related to a well-known algebra, the {\it splitting algebra} $A_{\Gamma}$. First introduced by Gelfand, Retakh, Serconek, and Wilson, splitting algebras originated from the problem of factoring non-commuting polynomials. Given a finite ranked poset $\Gamma$, we ask: Is $R_{\Gamma}$ Koszul? The Koszulity of $R_{\Gamma}$ is related to a combinatorial topology property of $\Gamma$ called Cohen-Macaulay. Kloefkorn and Shelton proved that if $\Gamma$ is a finite ranked cyclic poset, then $\Gamma$ is Cohen-Macaulay if and only if $\Gamma$ is uniform and $R_{\Gamma}$ is Koszul. We define a new generalization of Cohen-Macaulay, weakly Cohen-Macaulay, and we note that this new class includes posets with disconnected open subintervals. We prove: if $\Gamma$ is a finite ranked cyclic poset, then $\Gamma$ is weakly Cohen-Macaulay if and only if $R_{\Gamma}$ is Koszul. 
\end{abstract}
\date{April 1, 2016}

\maketitle


\baselineskip18pt

\section{Introduction}\label{intro}

We fix a field $\F$. Let $\Gamma$ denote a finite ranked poset with unique minimal element $*$ and order $<$. If $x, y \in \Gamma$, we write $x \rightarrow y$ if $x$ covers $y$.
\begin{defn}\label{RGammadef}The algebra $R_{\Gamma}$ is the graded quadratic $\F$-algebra with degree one generators $r_x$ for all $x \in \Gamma \setminus \{*\}$  and relations
\[ r_x \sum_{x \rightarrow y} r_y = 0 \hspace{.3in} \text{and} \hspace{.3in} r_x r_w = 0 \text{ whenever } x \not \rightarrow w.   \]
\end{defn}

The algebra $R_{\Gamma}$ has a relatively complex history -- we give a brief summary. In \cite{GRSW}, Gelfand, Retakh, Serconek and Wilson associate to $\Gamma$ a connected graded $\F$-algebra $A_{\Gamma}$, which is called the \emph{splitting algebra} of $\Gamma$; splitting algebras are related to the problem of factoring non-commuting polynomials. Retakh, Serconek and Wilson later showed that if $\Gamma$ satisfies a combinatorial condition called \emph{uniform}, then an associated graded algebra of the splitting algebra is quadratic, and it follows that $A_{\Gamma}$ is quadratic (c.f. \cite{RSW5}). These authors then asked a standard question in homological algebra: given a finite uniform ranked $\Gamma$, is $A_{\Gamma}$ Koszul?

\begin{defn}\label{koszuldef}
A connected graded $\F$-algebra $A$ is {\it Koszul} if the trivial right $A$-module $\F_A$ admits a linear projective resolution.
\end{defn} 

\noindent We note that Koszul algebras are quadratic and that there are many equivalent definitions (c.f. \cite{PP}).

If $A_{\Gamma}$ is Koszul, one can use the Hilbert series condition of numerical Koszulity to extract combinatorial data from the algebra. Recent work in the area of splitting algebras often focuses on calculating Hilbert series (c.f. \cite{RSW4} and \cite{RSW1}).

Given a specific $\Gamma$, it is difficult to determine if $A_{\Gamma}$ is Koszul. In fact, preliminary literature incorrectly asserted that $A_{\Gamma}$ is Koszul for all uniform $\Gamma$. We thus pass to a related question. Following \cite{RSW5} and assuming $\Gamma$ is uniform, we filter $A_{\Gamma}$ by rank in $\Gamma$. We denote the associated graded algebra by $grA_{\Gamma}$. Finally, we study the quadratic dual of $grA_{\Gamma}$, which is denoted by $(grA_{\Gamma})^!$. Applying standard techniques, we know that if $(grA_{\Gamma})^!$ is Koszul, then so is $A_{\Gamma}$. We then ask: given a finite uniform ranked $\Gamma$, is $(grA_{\Gamma})^!$ Koszul? We note in \cite{RSW4} and subsequent papers by the same authors, $(gr A_{\Gamma})^!$ is denoted by $B(\Gamma)$.

If $\Gamma$ is uniform, then $R_{\Gamma} = (grA_{\Gamma})^!$. The notation $R_{\Gamma}$ is from \cite{CPS}; Cassidy, Phan and Shelton assume $\Gamma$ is uniform and denote $(gr A_{\Gamma})^!$ with $R_{\Gamma}$. We emphasize: for 
Definition \ref{RGammadef}, $\Gamma$ need not be uniform. 

We are interested in the algebra $R_{\Gamma}$ and Koszulity of $R_{\Gamma}$, even if $\Gamma$ is not uniform and we draw no conclusions about splitting algebras. In \cite{CPS}, Cassidy, Phan and Shelton show that there exists a non-Koszul $R_{\Gamma}$. Also, if $\Gamma$ stems from a geometric object, then the Koszulity of $R_{\Gamma}$ gives important combinatorial and topological data; for example, the authors show that if $\Gamma$ is the intersection poset of a regular CW complex, then $R_{\Gamma}$ is Koszul. Sadofsky and Shelton study posets associated to regular CW complexes in \cite{SadSh} -- they show Koszulity for $R_{\Gamma}$ is a topological invariant.

Kloefkorn and Shelton found connections between $R_{\Gamma}$ and a combinatorial topology property known as Cohen-Macaulay. We say $\Gamma$ is cyclic if it has a unique maximal element. If $a < b$ in $\Gamma$, then $\Delta((a, b))$ denotes the order complex of $(a, b)$.

\begin{defn}\label{CMdef}
A finite ranked cyclic poset $\Gamma$ is {\it Cohen-Macaulay} relative to $\F$ if for all $a < b$ in $\Gamma$, $\tilde{H}^n(\Delta((a, b)), \F) = 0$ for all $n \neq dim \Delta((a, b))$. 
\end{defn}

\noindent A finite ranked cyclic poset is Cohen-Macaulay if every open interval is, as we say, a cohomology bouquet of spheres (or CBS). The Cohen-Macaulay property has been studied extensively; for a survey of all things Cohen-Macaulay, see \cite{BGSsurvey}. 

It is worth noting that both the Cohen-Macaulay and Koszul properties are relative to the field $\F$.

In \cite{KS}, Kloefkorn and Shelton proved the following theorem.

\begin{thm}\label{KSthm}
Let $\Gamma$ be finite ranked cyclic poset. Then $\Gamma$ is Cohen-Macaulay if and only if $\Gamma$ is uniform and $R_{\Gamma}$ is Koszul. 
\end{thm}

In this paper we define a new generalization of the Cohen-Macaulay property: weakly Cohen-Macaulay. The main theorem is as follows.

\begin{thm}\label{chap4thm}
Let $\Gamma$ be finite ranked cyclic poset. Then $\Gamma$ is weakly Cohen-Macaulay if and only if $R_{\Gamma}$ is Koszul. 
\end{thm}

This paper is organized as follows. Section \ref{defprelim} gives important background information. Sections \ref{genrightann} through \ref{anise} build the necessary machinery for our main theorem. Section \ref{weaklyCM} gives the definition of weakly Cohen-Macaulay and the statement of our main theorem. Finally, Section \ref{examples2} gives important examples and remarks.  

 
\section{Definitions and preliminaries}\label{defprelim}

\begin{defn}  
Let $\Gamma$ be a poset with unique minimal element $*$ and strict order $<$.  We say $\Gamma$ is {\it ranked} if for all $b\in \Gamma$, any two maximal chains in $[*,b]$ have the same length.  The length of such a maximal chain is then referred to as the  {\it rank} of $b$ and written $rk_\Gamma(b)$. Set $\Gamma_+ = \Gamma \setminus \{*\}$. Let $\Gamma(k)$ be the elements of $\Gamma$ of rank $k$.  

(1) $\Gamma$ is pure of rank $d$ if $rk_\Gamma(x)= d$ for every maximal element of $\Gamma$.  

(2) If $\Gamma$ is pure, then $\overline{\Gamma}$ is the poset $\Gamma$ adjoined with a unique maximal element.

(3) If $\Gamma$ is pure, then $\Gamma'$ is the poset $\Gamma \setminus \Gamma(rk(\Gamma))$.

(4) $\Gamma_x$ denotes the interval $[*,x]$ in $\Gamma$.

(5) $\Gamma$ is cyclic if $\Gamma = \Gamma_x$ for some $x\in \Gamma$.

(6) For any $x\in \Gamma$, $S_x(k) = \{ y\in \Gamma_x\, |\, rk_\Gamma(y) = rk_\Gamma(x) -k \}$.

(7) For each $k \geq 0$ we set $\Gamma^{> k} = \{y \in \Gamma | rk_{\Gamma}(y) > k\} \cup \{*\}$.

\noindent For any $a<b$, we say that $b$ covers $a$, written $b\to a$, if the closed interval $[a,b]$ has order 2, or equivalently $a\in S_b(1)$. This makes $\Gamma$ into a directed graph that is often referred to as a {\it layered} graph.
\end{defn}

The above definition of ranked poset is taken from \cite{GRSW}, a fundamental paper in the area of splitting algebras. We note that this definition differs from the traditional one wherein every maximal chain has the same length (c.f. \cite{StanleyEC}). 

From now forward, $\Gamma$ denotes a finite ranked poset and we will assume that $\Gamma$ has unique minimal element $*$ (without explicitly stating so). 

We recall the definition of {\it uniform} from \cite{GRSW}.

\begin{defn}\label{uniform} 
For $x\in \Gamma_+$ and $a,b \in S_x(1)$, write
$a\sim_x b$ if there exists $c\in S_a(1) \cap S_b(1)$ and extend $\sim_x$ to an equivalence relation on $S_x(1)$.   We say 
$\Gamma$ is {\it uniform} if, for every $x\in \Gamma_+$, $\sim_x$ has a unique equivalence class.  
\end{defn}

The {\it order complex} of a finite poset is a standard tool in combinatorial topology and we refer the reader to \cite{KS} for basic definitions. We denote the order complex of $\Gamma$ by $\Delta(\Gamma)$. The order complex of $\Gamma$ is a simplicial complex and its {\it geometric realization}, or {\it total space} is denoted $||\Delta(\Gamma)||$. All cohomology groups are all calculated with coefficients in our base field, $\F$.  


Recall from Definition \ref{CMdef} that a finite ranked cyclic poset is Cohen-Macaulay relative to a field $\F$ if the order complex of any open subinterval $(a,b)$ has non-zero reduced cohomology only in the degree equal to its dimension. Before giving an equivalent definition of Cohen-Macaulay, we recall notation from \cite{RSW4}.

\begin{defn} For any $a\in \Gamma$ and $1\le i\le rk_\Gamma(a)$ we set 
$$\Gamma_{a,i} = \{ w<a \,|\, rk_\Gamma(a)-rk_\Gamma(w) \le i-1\} = \Gamma^{>rk_\Gamma(a)-i} \cap (*,a)$$
\end{defn}

\noindent From \cite{KS}, we note: $\Gamma_{a,i}$ is a subposet of $[*,a)$ and the dimension of $\Delta(\Gamma_{a,i})$ is 
$i-2$.  Also, $\Gamma_{a,1} = \emptyset$, $\Gamma_{a,2} = S_1(a)$ and $\Gamma_{a,rk_\Gamma(a)} = (*,a)$.

From \cite{SheltonII}, we now give an equivalent definition of uniform.

\begin{prop}
Let $\Gamma$ be a finite ranked poset. Then $\Gamma$ is uniform if and only if for all $x \in \Gamma_+$ of rank at least three, $\Gamma_{x, 3}$ is connected as a graph.
\end{prop}

Similarly, we give an equivalent definition of Cohen-Macaulay. This statement and its proof are slightly different from those in \cite{SheltonII}. 

\begin{prop}\label{equivCM}
Let $\Gamma$ be a finite ranked cyclic poset. Then $\Gamma$ is Cohen-Macaulay if and only if for all $x \in \Gamma_+$ and all $rk_{\Gamma}(x) \geq k > n$, $\tilde{H}^{n-2}(\Delta(\Gamma_{x, k})) = 0$.
\end{prop}

\begin{proof}
Assume $\Gamma$ is Cohen-Macaulay and let $x \in \Gamma_+$. By Theorem \ref{KSthm}, $R_{\Gamma}$ is Koszul, from which it follows that $R_{\Gamma_x}$ is Koszul. Applying Theorem 2.7 from \cite{KS}, we see $H_{\Gamma_x}(m, j) =0$ for all $rk_{\Gamma}(x) > m > j \geq 0$. Theorem 4.2 together with Lemma 6.3 from \cite{KS} tell us that $H_{\Gamma_x}(m, j) = H^{m-j}(\Delta(\Gamma_{x, rk_{\Gamma}(x) - j}))$ for all $rk_{\Gamma}(x) > m \geq j \geq 0$, which completes the proof of the forward direction.

For the reverse direction, we proceed by induction on the rank of $\Gamma = \Gamma_b$. For a poset of rank one, there is nothing to show; we assume $\Gamma$ has rank $d+1$ with $d > 0$. 

By the inductive hypothesis, $\Gamma_a$ is Cohen-Macaulay for all $a < b$ in $\Gamma$. Let $m \leq rk_{\Gamma}(b)$. We observed that $\Delta(\Gamma_{b, m-1})$ is a closed subspace of $\Delta(\Gamma_{b, m})$. We obtain the standard long exact sequence
\[ \cdots \rightarrow \tilde{H}^n(\Delta(\Gamma_{b, m-1})) \rightarrow \tilde{H}^{n+1}(\Delta(\Gamma_{b, m}), \Delta(\Gamma_{b, m-1})) \rightarrow \tilde{H}^{n+1}(\Delta(\Gamma_{b, m})) \rightarrow \cdots. \]
Then, by assumption, $\tilde{H}^{n}(\Delta(\Gamma_{b, m}), \Delta(\Gamma_{b, m-1})) = 0$ for all $n \leq m - 2$. Similar to Lemma 6.4 from \cite{KS}, we get
\[ H^n(\Delta(\Gamma_{b, m}), \Delta(\Gamma_{b, m-1})) = \bigoplus_{a \in S_b(m-1)} \tilde{H}^{n-1}(\Delta((a, b))). \]
This implies
\[ \bigoplus_{a \in S_b(m-1)} \tilde{H}^{j}(\Delta((a, b))) = 0 \]
for all $j \leq m-3$. This implies $\Gamma$ is Cohen-Macaulay. 
\end{proof}

We now turn our attention to $R_{\Gamma}$ form Definition \ref{RGammadef}. Cassidy, Phan, and Shelton proved the following lemma in \cite{CPS}. This lemma is a very powerful tool; we will use it repeatedly and without further comment. We note that $\Gamma$ need not be uniform. 

\begin{lemma}[\cite{CPS}, (3.1)]\label{strongideal}
Let $\Gamma$ be a finite ranked poset. Then
\[ (R_{\Gamma})_+ = \bigoplus_{x \in \Gamma_+} r_xR_{\Gamma}. \]
\end{lemma}

We adopt the following notation and conventions from \cite{KS}. 

\begin{defn}[\cite{KS}, (2.5)]\label{dGamma} Let $\Gamma$ be a finite ranked poset. 

(1) $d_\Gamma = \sum\limits_{*\ne x\in \Gamma} r_x \in R_{\Gamma,1}$.  Also let $d_\Gamma$ denote the function 
$d_\Gamma: R_\Gamma \to R_\Gamma$ given by left (but never right) multiplication by $d_\Gamma$.  

(2) For all $n\ge k\ge 0$, set $R_\Gamma(n,k) = \sum\limits_{rk_\Gamma(y) = n+1} r_y R_{\Gamma,n-k}$.

\end{defn}

We note that the space $R_\Gamma(n,k)$ denotes the span of such monomials with $rk_\Gamma(b_1) = n+1$ and $rk_\Gamma(b_j) = k+1$ and the degree of such a monomial is $n-k+1$. Since $(d_\Gamma)^2 = 0$, for each $k\ge 0$ we have a cochain complex:
$$ \cdots R_\Gamma(n-1,k) \xrightarrow{d_\Gamma} R_\Gamma(n,k)  \xrightarrow{d_\Gamma} R_\Gamma(n+1,k) \cdots. $$




\section{The Right Annihilator Condition}\label{genrightann}

Throughout the remainder of this paper $\Gamma$ denotes an arbitrary finite ranked poset with unique minimal element $*$ and rank $m+1$.

We remind the reader of the following fact from Section 3 in \cite{CPS}: for a uniform $\Gamma$, the Koszulity of $R_\Gamma$ is equivalent to a condition on right annihilators of certain elements of $R_{\Gamma}$. In this section, we establish a more general result. We first need some notation and some new combinatorial objects.

\begin{defn}
Let $u \in R_{\Gamma}$. Define the linearly generated right annihilator of $u$, $ L_{\Gamma}(u) := (rann_{R_{\Gamma}}(u))_1R_{\Gamma}$. 
\end{defn}

\begin{defn}\label{r_Snotation}
Let $1 \leq n \leq m+1$ and $W \subset \Gamma(n)$. Then
\[ r_W := \sum_{z \in W} r_z. \] 
For all $x \in \Gamma_+$ we write $r_{\{x\}} = r_x$. Suppose $u, v \in \Gamma_+$. We write $u \sim^W v$ if for all
\[ q =  \sum_{z \in \Gamma_+} q_z r_z, \hspace{.3in} (q_z \in \F) \] 
$r_Wq = 0$ implies $q_u = q_v$. 
\end{defn}

We note $\sim^W$ defines an equivalence relation on $\Gamma_+$. Equivalence classes will be denoted by $[ - ]^W$ and thus
\[ r_{[z_0]^W} = \sum _{z \in [z_0]^W} r_z. \] 

\begin{prop}\label{Ldefinition}
Let $n$ and $W$ be given as in Definition \ref{r_Snotation}. Then  
\[ L_{\Gamma}(r_W) = \bigoplus_{[z_0]^W} r_{[z_0]^W} R_{\Gamma}. \] 
\end{prop}

\begin{proof}
Suppose
\[ r_W \cdot \sum_{z \in \Gamma_+} q_zr_z = 0. \]
Then 
\[ \sum_{z \in \Gamma_+} q_z r_z = \sum_{[z_0]^W} q_{[z_0]^W} r_{[z_0]^W} \in\bigoplus_{[z_0]^W} r_{[z_0]^W} R_{\Gamma}. \]
We also need to show $r_{[z_0]^W} \in rann_{R_\Gamma}(r_W)$ for all $[z_0]^W$. If $z_0 \in \Gamma \setminus \Gamma(n-1)$, then $r_{[z_0]^W} = r_{z_0}$ and $r_Wr_{z_0} = 0$. Now assume $z_0 \in \Gamma(n-1)$. We partition $W$ into two sets: $W_1 = \{ s \in W : s \rightarrow y \text{ for some } y \in [z_0]^W\}$ and $W_2$ is the compliment of $W_1$ in $W$. Then we compute
\[ r_W \cdot r_{[z_0]^W} = (r_{W_1} + r_{W_2}) r_{[z_0]^W} = r_{W_1}r_{[z_0]^W} = 0 \]
and this completes our proof. 
\end{proof}

\begin{defn}\label{ABHdef}
Let $n$ and $W$ be given as in Definition \ref{r_Snotation}. Let $U = \{u \in \Gamma(n-1) : u < s \text{ for some } s \in W\}$.  We define 
\[ A_{\Gamma}(r_W) := \bigoplus_{[z_0]^W \subseteq U} r_{[z_0]^W} R_{\Gamma}, \hspace{.4in} B_{\Gamma}(r_W) := \bigoplus_{z \in \Gamma(n-1) \setminus U} r_z R_{\Gamma} \]
and
\[ H_{\Gamma}(n-1) := \bigoplus_{z \in \Gamma_+ \setminus \Gamma(n-1)} r_zR_{\Gamma}. \]
\end{defn}

We observe:
\[ L_{\Gamma}(r_W) = A_{\Gamma}(r_W) \oplus B_{\Gamma}(r_W) \oplus H_{\Gamma}(n-1). \] 

\begin{defn}
Assume $\Gamma = \Gamma_x$. We define
\[ T_{m+1}(\Gamma) = \{ \Gamma(m+1) \} = \{ \{x\} \} \]
and recursively define for $m \geq i \geq 1$
\[ T_i(\Gamma) = \{[z_0]^W \subset \Gamma_+ | W \in T_{i+1} \text{ and there exists } s \in W \text{ such that } s \rightarrow z_0 \}. \]  
Then define
\[ T(\Gamma) = \bigcup_{i=1}^{m+1} T_i(\Gamma). \]
For an arbitrary finite ranked poset, we define
\[ \mathcal{T}(\Gamma) = \bigcup_{x \in \Gamma_+} T(\Gamma_x). \]
\end{defn}

\begin{rmk}\label{treerem} 
By (3.3) from \cite{CPS}, we see that if $\Gamma$ is uniform, then $\mathcal{T}(\Gamma)$ consists of sets of the form $\Gamma_x(n)$ where $x \in \Gamma_+$ and $1 \leq n \leq rk_{\Gamma}(x)$. 
\end{rmk}

The following lemma is (3.3) of \cite{CPS}, without the uniform hypothesis.

\begin{thm}\label{rightannihilator}
Let $\Gamma$ be a finite ranked poset. The algebra $R_{\Gamma}$ is Koszul if and only if for all $W \in \mathcal{T}(\Gamma)$, $rann_{R_{\Gamma}}(r_W) = L_{\Gamma}(r_W)$.
\end{thm}

\begin{proof}
For the forward direction, assume $R_{\Gamma}$ is Koszul. Then the right $R_{\Gamma}$-modules $(R_{\Gamma})_+$ and $\F_{R_{\Gamma}} = R_{\Gamma}/(R_{\Gamma})_+$ are Koszul. Let $W \in \mathcal{T}(\Gamma)$, so there is $x \in \Gamma_+$ and $1 \leq n \leq rk_{\Gamma}(x)$ with $W \in T_n(\Gamma_x)$. By (reverse) induction on $n$, we will simultaneously prove $rann_{R_{\Gamma}}(r_W) = L(r_W)$ and $r_WR_{\Gamma}$ is a Koszul module.

If $n = rk_{\Gamma}(x)$, then $W = \{x\}$. By definition of $R_{\Gamma}$, $rann_{R_{\Gamma}}(r_x) = L_{\Gamma}(r_x)$. $r_xR_{\Gamma}$ is a direct summand of $(R_{\Gamma})_+$ by Lemma \ref{strongideal}, thus it is a Koszul module. 

We now assume $n < rk_{\Gamma}(x)$. Then there exists $U \in T_{n+1}(\Gamma_x)$, $z_0 \in \Gamma_x(n)$ and $u \in U$ such that $u \rightarrow z_0$ and $W = [z_0]^U$.  By induction, $rann_{R_{\Gamma}}(r_U) = L(r_U)$. We then have a short exact sequence
\[ 0 \rightarrow L_{\Gamma}(r_U) \rightarrow R_{\Gamma} \rightarrow r_UR_{\Gamma} \rightarrow 0. \]
Again, by induction, $r_UR_{\Gamma}$ is Koszul. Thus $L_{\Gamma}(r_U)$ is a Koszul module. The module $r_WR_{\Gamma}$ is a direct summand of $L(r_U)$ and therefore, $r_WR_{\Gamma}$ is a Koszul module. Now $rann_{R_{\Gamma}}(r_W)$ is linearly generated and $rann_{R_{\Gamma}}(r_W) = L_{\Gamma}(r_W)$. This completes the proof of the forward direction.

We now prove the reverse direction. Let $\mathcal{F}$ be the set of right ideals in $R_{\Gamma}$ of the form
\[ I = \bigoplus_{i=1}^m r_{W_i} R_{\Gamma} \]
where $W_1, \ldots, W_m \in \mathcal{T}(\Gamma)$ are pairwise disjoint. We include the zero ideal in $\mathcal{F}$ and observe $(R_{\Gamma})_+ \in \mathcal{F}$. To prove that $R_{\Gamma}$ is Koszul, we show that $\mathcal{F}$ is a Koszul filtration (c.f. \cite{Pi}).

Let 
\[ 0 \neq I = \bigoplus_{i=1}^m r_{W_i} R_{\Gamma}  \in \mathcal{F}. \]
Then set
\[ J = \bigoplus_{i=1}^{m-1} r_{W_i} R_{\Gamma} \]
and observe $J \in \mathcal{F}$ and $I = J + r_{W_m}R_{\Gamma}$. Also, the right ideal $(r_{W_m} : J) = \{p \in R_{\Gamma} : r_{W_m}p \in J \}$ (the conductor of $r_{W_m}$ into $J$) is equal to $rann_{R_{\Gamma}}(r_{W_m})$. Since $rann_{R_{\Gamma}}(r_{W_m}) = L_{\Gamma}(r_{W_m})$, $(r_{W_m} : J) \in \mathcal{F}$. We conclude $\mathcal{F}$ is a Koszul filtration, which completes our proof. \end{proof}



The following lemma is (3.5) from \cite{CPS}, again, without the uniform hypothesis.

\begin{thm}\label{cyclicgen}
Let $\Gamma$ be a finite ranked poset. Then $R_{\Gamma}$ is Koszul if and only if $R_{\Gamma_x}$ is Koszul for all $x \in \Gamma_+$. 
\end{thm}

\begin{proof}
Let $x \in \Gamma_+$ and $W \in T(\Gamma_x)$. By Lemma \ref{strongideal},
\[rann_{R_{\Gamma}}(r_W) \cap R_{\Gamma_x} = rann_{R_{\Gamma_x}}(r_W). \]
The theorem then follows from Theorem \ref{rightannihilator}. 
\end{proof}

\begin{cor}
Let $\Gamma$ be a finite ranked poset with rank less than or equal to three. Then $R_{\Gamma}$ is Koszul. 
\end{cor}

\begin{rmk}\label{importantrmk}
We often use the following fact, which is Theorem \ref{rightannihilator} together with Theorem \ref{cyclicgen}. Suppose $\Gamma$ is a finite ranked cyclic poset with $\Gamma = \Gamma_x$. Set $\Omega = \Gamma \setminus \{x\}$ and assume $R_{\Omega}$ is Koszul. Then $R_{\Gamma}$ is Koszul if and only if for all $W \in T(\Gamma)$, $rann_{R_{\Gamma}}(r_W) = L_{\Gamma}(r_W)$. 
\end{rmk}

We thus have a generalized right annihilator condition, but we wish to refine it. In order to do so, we first define additional notation for $\Gamma$. 

\begin{defn}\label{nSnotation}
Let $1 \leq n \leq m+1$ and $W \subset \Gamma(n)$. We set 
\[ \Gamma_W = \{a \in \Gamma : a \leq s \text{ for some } s \in W\} = \bigcup_{s \in W} [*, s]. \]
Also, for $0 \leq k \leq n-1$, we set
\[ \Gamma(W, k) =  \Gamma_W \cap \Gamma^{\geq n-k}.\]
We say $W$ is {\it linked} if $\Gamma(W, 1)$ is connected as a graph. We say $W' \subseteq W$ is {\it maximally linked} relative to $W$ if $W'$ is maximal amongst linked subsets of $W$. 
\end{defn}

The collection of all maximally linked subsets of $W$ forms a partition of $W$.

\begin{lemma}\label{ralem1}
Let $n$ and $W$ be given as in Definition \ref{nSnotation}. Then $rann_{R_{\Gamma}}(r_W) = L_{\Gamma}(r_W)$ if and only if for all maximally linked subsets $W' \subseteq W$, $rann_{R_{\Gamma}}(r_{W'}) = L_{\Gamma}(r_{W'})$.
\end{lemma}

\begin{proof}
Due to the observation immediately following Definition \ref{ABHdef}, it is enough to prove the lemma for $\Gamma = \Gamma_W$. We therefore assume $W = \Gamma(n)$. We enumerate the maximally linked subsets of $W$: $W_1, \ldots W_p$. For all $i = 1, \ldots, p$, set $U_i = \{t \in \Gamma(n-1) : t < s \text{ for some } s \in W_i\}$. 

For the forward direction, suppose towards a contradiction that for some $k \in \{1, \ldots, p\}$, there exists $A \in rann_{R_{\Gamma}}(r_{W_k}) \setminus L_{\Gamma}(r_{W_k})$. Then $A = A_1 + A_2$ with 
\[ A_1 \in \bigoplus_{z \in U_k} r_z R_{\Gamma} \hspace{.2in} \text{and} \hspace{.2in} A_2 \in \bigoplus_{z \in \Gamma_+ \setminus U_k} r_z R_{\Gamma} = B_{\Gamma}(r_{W_k}) \oplus H_{\Gamma}(n-1). \]
Clearly $r_{W_k}A_2 = 0$, thus $r_{W_k}A_1 = 0$. Also, $A_1$ is nonzero and $A_1 \not \in r_{U_k}R_{\Gamma} = A_{\Gamma}(r_{W_k})$; if not, then $A \in L_{\Gamma}(r_{W_k})$. We compute
\[ r_WA_1 = (r_{W_k} + r_{\Gamma(n) \setminus W_k}) A_1 = r_{W_k}A_1 = 0. \]
Therefore, $A_1 \in rann_{R_{\Gamma}}(r_W)$. Also, we see $A_1 \not \in A_{\Gamma}(r_W) = \bigoplus_{i=1}^p r_{U_i}R_{\Gamma}$ since $A_1 \not \in r_{U_k}R_{\Gamma}$. Thus $A_1 \not \in L_{\Gamma}(r_W)$, which is a contradiction.
 
For the converse, we assume $rann_{R_{\Gamma}}(r_{W_i}) = L_{\Gamma}(r_{W_i})$ for all $i = 1, \ldots, p$. Carefully applying Definition \ref{ABHdef}, we compute
\begin{align*}
rann_{R_{\Gamma}}(r_W) & = \bigcap_{i=1}^p rann_{R_{\Gamma}}(r_{W_i}) \\
& = \bigcap_{i=1}^p L_{\Gamma}(r_{W_i}) \\
& = \bigoplus_{i=1}^p A_{\Gamma}(r_{W_i}) \oplus H_{\Gamma}(n-1) \\
& = \bigoplus_{[z_0]^W} r_{[z_0]^W} R_{\Gamma} \\
& = L_{\Gamma}(r_W),
\end{align*}
which completes our proof.
\end{proof}

\begin{lemma}\label{ralem2}
Let $n$ and $W$ be given as in Definition \ref{nSnotation}. Assume $W$ is linked. Then $rann_{R_{\Gamma}}(r_W) = L_{\Gamma}(r_W)$ if and only if for all $0 \leq k \leq n-3$, $H^{n-2}(R_{\Gamma_W}( \bullet, k), d_{\Gamma_W}) = 0$. 
\end{lemma}

\begin{proof}
Again, it is enough to prove the claim for $\Gamma = \Gamma_W$. We set $W = \Gamma(n)$. Then $rann_{R_{\Gamma}}(r_W) = L_{\Gamma}(r_W)$ if and only if 
\[ rann_{R_{\Gamma}}(r_W) = r_{\Gamma(n-1)}R_{\Gamma} \oplus H_{\Gamma}(n-1). \]
This equality holds if and only if 
\begin{diagram}
R_{\Gamma}(n-3, k) & \rTo^{r_{\Gamma(n-1)} \cdot} & R_{\Gamma}(n-2, k) & \rTo^{r_W \cdot} &  R_{\Gamma}(n-1, k) 
\end{diagram}
is exact for all $0 \leq k \leq n-3$. 
\end{proof}

\begin{defn}
Assume $\Gamma = \Gamma_x$. We define
\[ M_n(\Gamma) = \bigcup_{W \in T_n(\Gamma)} \{W' : W' \text{ is maximally linked relative to } W\}  \]
and
\[ M(\Gamma) = \bigcup_{n=1}^{m+1} M_n(\Gamma). \]
\end{defn}

\begin{rmk}\label{Muniform}
Using Remark \ref{treerem}, we see that if $\Gamma$ is uniform then $M_n(\Gamma) = \{ \Gamma(n) \}$ for all $1 \leq n \leq m+1$.
\end{rmk}

We combine the definition of $M(\Gamma)$ with Lemmas \ref{ralem1} and \ref{ralem2} and Remark \ref{importantrmk} for the following theorem about the Koszulity of $R_{\Gamma}$. 

\begin{thm}\label{rightannihilator2}
Let $\Gamma$ be a finite ranked poset and assume $\Gamma = \Gamma_x$. Set $\Omega = \Gamma \setminus \{x\}$ and assume $R_{\Omega}$ is Koszul. The following are equivalent:
\begin{enumerate}
\item $R_{\Gamma}$ is Koszul;
\item for all $W \in M(\Gamma)$, $rann_{R_{\Gamma}}(r_W) = L_{\Gamma}(r_W)$;
\item for all $1 < n \leq m$, $W \in M_n(\Gamma)$ and $0 \leq k \leq n-3$, $H^{n-2}(R_{\Gamma_W}( \bullet, k), d_{\Gamma_W}) = 0$.
\end{enumerate}
\end{thm}

We end this section with an example that motivates the definition of weakly Cohen-Macaulay and our main theorem.

\begin{example}\label{ex4.1}
Let $\Theta$ be the ranked poset shown in Figure \ref{theta}.

\begin{figure}[h!]
\centering
\begin{minipage}{.5\textwidth}
  \centering
 \includegraphics[scale=0.5]{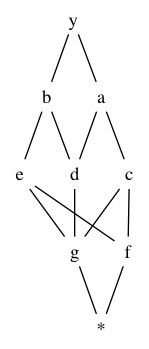}
  \caption{The poset $\Theta$.}
  \label{theta}
\end{minipage}%
\begin{minipage}{.5\textwidth}
  \centering
 \includegraphics[scale=0.5]{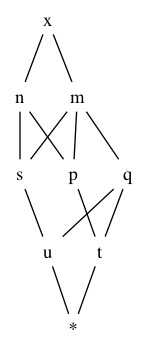}
  \caption{The poset $\Theta^*$.}
  \label{theta*}
\end{minipage}
\end{figure}

This poset $\Theta$ was first introduced by Cassidy and was studied extensively in \cite{Nacin}. By inspection, we see $\Theta$ is uniform. Also, by inspection, we know $\Theta$ is not Cohen-Macaulay; $ \| \Delta(\Theta_{y, 4}) \|$ is homotopic to $S^1$. By Theorem \ref{KSthm}, $R_{\Theta}$ is not Koszul.

Let $\Theta^*$ be the dual poset of $\Theta$, as shown in Figure \ref{theta*}. Since $\Theta^*_{n, 3}$ is disconnected as a graph, $\Theta^*$ is not uniform and  $\Theta^*$ is not Cohen-Macaulay. Also, $\| \Delta({\Theta^*}_{x, 4}) \| = \| \Delta(\Theta_{y, 4}) \|$ is homotopic to $S^1$. 

We claim $R_{\Theta^*}$ is Koszul. By Theorem \ref{rightannihilator2}, it suffices to show that $rann_{R_{\Theta^*}}(r_m + r_n)$ is linearly generated. We see $(r_s + r_p + r_q)r_u = r_qr_u = -r_qr_t$. Then, by inspection,
\[ rann_{R_{\Theta^*}}(r_m + r_n) = (r_s + r_p + r_q)R_{\Theta^*} \oplus H_{\Theta^*}(2). \] 
\end{example}


\section{A Spectral Sequence Associated to $\Delta(\Gamma \setminus \{*\})$ }\label{spectralsequence}

Let $Y = \| \Delta(\Gamma \setminus \{*\}) \|$. We start by defining an unusual filtration on $C^{\bullet}(Y)$. Let $F_pC^{\bullet}(Y)$ be an increasing filtration on $C^{\bullet}(Y)$ given by
\[ F_pC^n(Y) = \{ (\alpha_0, \ldots, \alpha_n) \in C^n(Y) : rk_{\Gamma}(\alpha_n) \geq (m+1) - p\}. \]
We note: $F_pC^{n}(Y)$ consists of all $n$ co-chains of $Y$ that emanate from the top $p+1$ layers of $\Gamma$. Also, observe
\[ 0 = F_{-1}C^{\bullet}(Y) \subset F_0C^{\bullet}(Y) \subset \cdots \subset F_{m-1}C^{\bullet}(Y) \subset F_{m}C^{\bullet}(Y) = C^{\bullet}(Y). \]

Let $E^n_{p,q}$ denote the associated cohomology spectral sequence with
\[ E^0_{p,q} = F_pC^{p+q}(Y)/F_{p-1}C^{p+q}(Y) \hspace{.3in} \text{and} \hspace{.3in} E^1_{p,q} = H^{p+q}(E^0_{p*}). \]
It is bounded since $E^0_{p,q} = 0$ if $p < 0$, $p+q < 0$, or $q > -2p + m$. The differentials are maps $d^r_E : E^r_{p,q} \rightarrow E^r_{p-r, q+r+1}$ induced by the differential $d_{Y}$ on $C^{\bullet}(Y)$. 

\begin{prop}
Let $\Gamma$ be a ranked poset. Let $E^r_{p, q}$ be the spectral sequence of the filtration $F_{\bullet}$ on $C^{\bullet}(Y)$. For all $p$ and $q$, $E^{m+1}_{p,q} = E^{\infty}_{p,q}$. Also $E^1_{p,q} \implies H^{p+q}(Y)$. 
\end{prop}

\begin{proof}
The equality holds because $d^{m+1}_E$ is the zero map on $E^{m+1}_{p,q}$. The convergence statement follows from Theorem 5.5.1 of \cite{CW}. 
\end{proof}

The $E^1$ page of our spectral sequence can be written in terms of reduced cohomology of open intervals of $\Gamma$. 

\begin{prop}
Let $\Gamma$ be a finite ranked poset. Let $E^r_{p, q}$ be the spectral sequence of the filtration $F_{\bullet}$ on $C^{\bullet}(Y)$. For all $p$ and $q$,
\[ E^1_{p,q} = \bigoplus_{rk_{\Gamma}(x) = (m+1)-p} \tilde{H}^{p+q-1}( \Delta((*, x))). \]
\end{prop}

\begin{proof}
We fix $p$ and identify two chain complexes: $(E^0_{p, *}, d^0_E)$ and 
\[ \left( \bigoplus_{rk_{\Gamma}(x) = (m+1)-p} \tilde{C}^{\bullet}( \Delta((*, x)) ), \bigoplus_{rk_{\Gamma}(x )= (m+1)-p}  d_{ \Delta((*, x)) } \right).  \] 
For convenience of notation, elements of $\tilde{C}^{\bullet}( \Delta((*, x)))$ will be denoted by sums of chains of the form $(\alpha_0, \ldots, \alpha_n)_x$. Define a map
\[f : E^0_{p, q} = F_pC^{p+q}(Y)/F_{p-1}C^{p+q}(Y) \rightarrow \bigoplus_{rk_{\Gamma}(x) = (m+1)-p} \tilde{C}^{p+q-1}( \Delta((*, x)) )\]
via $(\alpha_0) \mapsto (1)_{\alpha_0}$ and if $p + q > 0$
\[ (\alpha_0, \ldots, \alpha_{p+q}) \mapsto (\alpha_0, \ldots, \alpha_{p+q-1})_{\alpha_{p+q}}, \]
and extend linearly. It is easy to see that $f$ is bijective. We claim $f$ is a co-chain map. Let $(\alpha_0, \ldots, \alpha_{p+q}) \in E^0_{p, q}$ and note $rk_{\Gamma}(\alpha_{p+q}) = m+1 - p$. We compute
\begin{align*}
\bigoplus_{rk_{\Gamma}(x )= (m+1)-p}  d_{\Delta((*, x)) } \left( f(\alpha_0, \ldots, \alpha_{p+q}) \right) & = \bigoplus_{rk_{\Gamma}(x )= (m+1)-p}  d_{\Delta((*, x)) } \left( (\alpha_0, \ldots, \alpha_{p+q-1})_{\alpha_{p+q}} \right) \\
 & = d_{ \Delta((*, \alpha_{p+q})) } \left(\alpha_0, \ldots, \alpha_{p+q-1}  \right) \\
\end{align*}
and on the other hand
\begin{align*}
f(d^0_E(\alpha_0, \ldots, \alpha_{p+q})) & = f \left( \sum_{x < \alpha_0} ( x, \alpha_0, \ldots, \alpha_{p+q}) \right) \\
& + f \left (\sum_{i=0}^{p+q-1} (-1)^{i+1} \sum_{\alpha_i < x < \alpha_{i+1}} \left( \alpha_0, \ldots, \alpha_i,  x, \alpha_{i+1}, \ldots, \alpha_{p+q} \right) \right)  \\
 & = \sum_{x < \alpha_0} (x, \alpha_0, \ldots, \alpha_{p+q-1})_{\alpha_{p+q}} \\ 
 & + \sum_{i=0}^{p+q-1} (-1)^{i+1} \sum_{\alpha_i < x < \alpha_{i+1}} \left(\alpha_0, \ldots, \alpha_i, x, \alpha_{i+1}, \ldots, \alpha_{p+q-1} \right)_{\alpha_{p+q}} \\
& = d_{ \Delta((*, \alpha_{p+q})) } \left(\alpha_0, \ldots, \alpha_{p+q-1}  \right). 
\end{align*} It follows that $f$ is a co-chain isomorphism.
\end{proof}

We now make an observation related to Theorem \ref{KSthm}. 

\begin{cor}\label{interesting}
Let $\Gamma$ be a finite ranked cyclic poset and let $E^r_{p, q}$ be the spectral sequence of the filtration $F_{\bullet}$ on $C^{\bullet}(Y)$. Assume $\Gamma$ is Cohen-Macaulay. Then $E^1_{p, q} = 0$ unless $q = -2p + m$. Moreover, $E^2_{p, q} = E^{\infty}_{p, q}$. 
\end{cor}

The first page of our spectral sequence is important for our main theorem, so we will introduce some convenient notation. 

\begin{defn}
For $0 \leq n \leq m$, define
\[ S^n(\Gamma_+) = \bigoplus_{rk_{\Gamma}(x) = n+1} \tilde{H}^{n-1}( \Delta((*, x)) ). \]
For $n > 0$, we denote elements in $\tilde{H}^{n-1}( \Delta((*, x)) )$ with linear combinations of equivalence classes of the form $[x_1 \leftarrow \cdots \leftarrow x_n]_x$. This notation requires: $x_n \leftarrow x$. Elements in $\tilde{H}^{-1}(\Delta((*, x)) )$ will be denoted by scalar multiples of $[1]_x$. 

Define $d_{S^n(\Gamma_+)} : S^n(\Gamma_+) \rightarrow S^{n+1}(\Gamma_+)$ by extending linearly from the formula
\[ [1]_x \mapsto \bigoplus_{x \leftarrow y} [x]_y \]
and for $n > 0$
\[ [x_1 \leftarrow \cdots \leftarrow x_n]_x \mapsto \bigoplus_{x \leftarrow y} (-1)^n[x_1 \leftarrow \cdots \leftarrow x_n \leftarrow x]_y. \]
\end{defn}

It is worth noting that we will later study the spaces $S^n((\Gamma^{>k})_+)$ and $S^n(\Gamma(W, k))$, which is consistent with our above notation.  

\begin{prop}
The map $d_{S^n(\Gamma_+)} : S^n(\Gamma_+) \rightarrow S^{n+1}(\Gamma_+)$ is well-defined. 
\end{prop}

\begin{proof}
If $n = 0$, the map is clearly well-defined. Assume $n > 0$. It is sufficient to observe the inclusion
\begin{align*}
d_{S^n(\Gamma_+)} \left( im(d_{\Delta((*, x))}^{n-2}: C^{n-2}(\Delta((*, x))) \rightarrow C^{n-1}(\Delta((*, x)))) \right) \\
\subset im(d_{ \Delta((*, z))}^{n-1}: C^{n-1}(\Delta((*, z))) \rightarrow C^{n}(\Delta((*,z)))). \
\end{align*}
for all $x \leftarrow z$.
\end{proof}






\begin{prop}
$(S^{\bullet}(\Gamma_+), d_{S^{\bullet}(\Gamma_+)})$ is a co-chain complex.
\end{prop}

\begin{proof}
Let $[1]_x \in \tilde{H}^{-1}(\Delta((*, x)))$. Then
\begin{align*}
d_{S^{1}(\Gamma_+)} \circ d_{S^{0}(\Gamma_+)}([1]_x) & = d_{S^{1}(\Gamma_+)} \left( \bigoplus_{x \leftarrow y} [x]_y \right) \\
& = \bigoplus_{x < z} \left( \sum_{x \leftarrow y \leftarrow z} [x \leftarrow y]_z \right) \\
& =  \bigoplus_{x < z} [0]_{z}
\end{align*}
since
\[
d^{0}_{ \Delta((*, z))}(x) = -\sum_{x \leftarrow y \leftarrow z} (x,  y).
\]
for all $x < z$.

Assume $n > 0$. Let $[x_1 \leftarrow \cdots \leftarrow x_{n}]_x \in \tilde{H}^{n-1}(\Delta((*, x)))$. We compute
\begin{align*}
 & d_{S^{n+1}(\Gamma_+)} \circ d_{S^{n}(\Gamma_+)}([x_1 \leftarrow \cdots \leftarrow x_{n}]_x) \\
& = d_{S^{n+1}(\Gamma_+)} \left ( \bigoplus_{x \leftarrow y}  (-1)^{n}[x_1 \leftarrow \cdots \leftarrow x_{n} \leftarrow x]_y\right ) \\
& = \bigoplus_{x < z} \left( \sum_{x \leftarrow y \leftarrow z} (-1)^{2n+1} [x_1 \leftarrow \cdots \leftarrow x_{n} \leftarrow x \leftarrow y ]_z \right) \\
& =  \bigoplus_{x < z} [0]_{z}.
\end{align*}

The last equality holds because 

\[
d^{n}_{ \Delta((*, z))} \left(x_1, \ldots, x_{n}, x \right) = \sum_{x \leftarrow y \leftarrow z} (-1)^{n+1}\left( x_1, \ldots, x_{n}, x,  y  \right).
\]
for all $x < z$.
\end{proof}

We note $S^{p+q}(\Gamma_+)$ is $E^1_{p,q}$ for $q = -2p + m$. In fact, $d_{S^n(\Gamma_+))}$ is the differential $d^1_E$ on $E^{1}_{p, -2p + m}$.


\section{An Isomorphism of Co-chain Complexes}\label{anise}

We now relate Sections \ref{genrightann} and \ref{spectralsequence}. 

\begin{defn}
Define $\Psi_{\Gamma_+} : S^n(\Gamma_+) \rightarrow R_{\Gamma}(n, 0)$ via
\[ [1]_x \mapsto r_x \]
and for $n > 0$ and extend linearly from the formula
\[ [x_1 \leftarrow \cdots \leftarrow x_n]_x \mapsto r_xr_{x_n} \cdots r_{x_1}. \]
\end{defn}

\begin{prop}
The map $\Psi_{\Gamma_+} : S^{n}(\Gamma_+) \rightarrow R_{\Gamma}(n, 0)$ is well-defined.
\end{prop}

\begin{proof}
The map is clearly well-defined for $n= 0$. We then assume $n > 0$ and let $rk_{\Gamma}(x) = n+1$. Similar to the above proofs, it suffices to show: 

\[ \Psi_{\Gamma_+}  \left( im(d_{\Delta((*, x))}^{n-2}: C^{n-2}(\Delta(*, x)) \rightarrow C^{n-1}(\Delta((*, x)))) \right) = 0.  \]

Let $\alpha = (\alpha_1, \ldots, \alpha_{n-1}) \in C^{n-2}(\Delta((*, x)))$. Since $n-1 = dim \Delta((*, x))$, $[d^{n-2}_{\| \Delta((*, x)) \|}(\alpha)]_x$ is 

\[ \sum_{a \leftarrow \alpha_1} [ a \leftarrow \alpha_1 \leftarrow \cdots \leftarrow \alpha_{n-1} ]_x, \hspace{.3in} \sum_{\alpha_{n-2} \leftarrow a} [ \alpha_1 \leftarrow \cdots \leftarrow \alpha_{n-1} \leftarrow a ]_x, \] 
\[ \text{or} \hspace{.3in} \sum_{\alpha_i \leftarrow a \leftarrow \alpha_{i+1}} [ \alpha_1 \leftarrow \cdots \leftarrow \alpha_i \leftarrow a \leftarrow \alpha_{i+1} \leftarrow \cdots \leftarrow a_{n-1} ]_x \]
for some $i \in \{1, \ldots, n-2\}$. Using an observation made in the proof of Lemma 4.1 in \cite{KS}, we see that $\Psi_{\Gamma_+}$ evaluated at any of the above elements is zero in $R_{\Gamma}$. 
\end{proof}

The following theorem is an improvement of 3.2.1 of \cite{RSW4}. It is also an improvement of Theorem 4.2 from \cite{KS}. Unlike Theorem 4.2 in \cite{KS}, we need not assume $R_{\Gamma}$ is Koszul. 

\begin{thm}
$\Psi_{\Gamma_+} : S^{\bullet}(\Gamma_+) \rightarrow R_{\Gamma}(\bullet, 0)$ is a co-chain isomorpshim. 
\end{thm}

\begin{proof}
Let $[1]_x \in \tilde{H}^{-1}(\Delta((*, x)))$. Then
\[ \Psi_{\Gamma_+}(d_{S^0(\Gamma_+))}([1]_x) = \Psi_{\Gamma_+}\left( \bigoplus_{x \leftarrow y} [x]_y \right)
 = \left( \sum_{x \leftarrow y} r_y \right)r_x \]
and
\[ d_{\Gamma}(\Psi_{\Gamma_+}([1]_x) = d_{\Gamma}(x) = \left( \sum_{z \in \Gamma_+} r_z \right) r_x = \left( \sum_{x \leftarrow y} r_y \right)r_x. \]
Assume $n > 0$. Let $[x_1 \leftarrow \cdots \leftarrow x_n]_x \in \tilde{H}^{n-1}(\Delta((*, x)))$. We compute
\begin{align*}
\Psi_{\Gamma_+}(d_{S^n(\Gamma_+))}( [x_1 \leftarrow \cdots \leftarrow x_n]_x )) & = \Psi_{\Gamma_+} \left ( \bigoplus_{x \leftarrow y} (-1)^n[x_1 \leftarrow \cdots \leftarrow x_n \leftarrow x]_y \right) \\
 & = (-1)^n \left( \sum_{y \rightarrow x} r_y \right) r_xr_{x_n} \cdots r_{x_1}
\end{align*}
and
\[ d_{\Gamma}(\Psi_{\Gamma_+}( [x_1 \leftarrow \cdots \leftarrow x_n]_x )) = d_{\Gamma}(r_xr_{x_n} \cdots r_{x_1}) = \left( \sum_{y \rightarrow x} r_y \right) r_xr_{x_n} \cdots r_{x_1}. \] We conclude that $\Psi_{\Gamma}$ is a chain map. 

It is clear that $\Psi_{\Gamma_+}$ is an epimorphism. It remains to show that $\Psi_{\Gamma_+}$ is injective. We will proceed by induction on the rank of $\Gamma$. If $\Gamma$ has rank one ($m=0$), then

\[ dim_{\F} R_{\Gamma}(0, 0) = \|\Gamma_+\| = \sum_{x \in \Gamma_+} dim_{\F}\tilde{H}^{-1}(\Delta((*, x))) = dim_{\F}S^0(\Gamma) \]
which proves the base case. 

We now assume the theorem is true for all posets of rank $m$. It suffices to prove the result for cyclic posets of rank $m+1$, since 
\[ R_{\Gamma}(m, 0) = \bigoplus_{rk_{\Gamma}(z) = m+1} R_{\Gamma_z}(m, 0). \]
Thus, we set $\Gamma = \Gamma_x$ and note $S^m(\Gamma_+) = \tilde{H}^{m-1}(\Delta((*, x)))$. The poset $\Gamma_{< x}$ has rank $m$ and thus $\Psi_{(\Gamma_{<x})_+} : S^{\bullet}((\Gamma_{<x})_+) \rightarrow R_{\Gamma_{< x}}(\bullet, 0)$ is a co-chain isomorphism. It is apparent that $S^n(\Gamma_+) = S^n((\Gamma_{<x})_+)$ and $R_{\Gamma}(n, 0) = R_{\Gamma_{< x}}(n, 0)$ for $0 \leq n \leq m-1$. The maps $d_{S^n(\Gamma_+)}$ and $d_{S^n((\Gamma_{<x})_+)}$ coincide for $0 \leq n \leq m-2$. Similarly, the maps $d_{\Gamma}$ and $d_{\Gamma_{<x}}$ coincide for $0 \leq n \leq m-2$. Hence, to prove $\Psi_{\Gamma_+}$ is injective, it suffices to prove $\tilde{H}^{m-1}(\Delta((*, x))) \simeq R_{\Gamma_x}(m,0)$. 

We have a commutative diagram
\begin{diagram}
\bigoplus_{rk_{\Gamma}(y) = m} \tilde{H}^{m-2}(\Delta((*, y))) & \rTo^{d_{S^{m-1}(\Gamma_+)}} & \tilde{H}^{m-1}(\Delta((*, x))) \\
\dTo^{\Psi_{(\Gamma_{<x})_+}} && \dTo_{\Psi_{\Gamma_+}}\\
\bigoplus_{rk_{\Gamma}(y) = m} R_{\Gamma_y}(m-1, 0)  & \rTo_{d_{\Gamma}} & R_{\Gamma_x}(m,0) \\
\end{diagram}
By the inductive hypothesis, $\Psi_{(\Gamma_{<x})_+}$ is injective. We will use the notation $\alpha = (\alpha_1, \cdots, \alpha_m)$ for chains in $C^{m-1}(\Delta((*, x)))$. Suppose
\[ \sum_{\alpha \in C^{m-1}(\Delta((*, x)))} q_{\alpha}[\alpha_1 \leftarrow \cdots \leftarrow \alpha_m]_x \in \tilde{H}^{m-1}(\Delta((*, x))) = S^m(\Gamma_+) \]
is in the kernel of $\Psi_{\Gamma_+}$. Thus
\[ \sum _{\alpha \in C^{m-1}(\Delta((*, x)))} q_{\alpha}r_xr_{\alpha_m}\cdots r_{\alpha_1} = r_x \left(  \sum _{\alpha \in C^{m-1}(\Delta((*, x)))} q_{\alpha}r_{\alpha_m}\cdots r_{\alpha_1}   \right) = 0.  \]
We set $\Omega = (\Gamma_{< x})'$. Thus, by definition of $R_{\Gamma}$,
\[ \sum _{\alpha \in C^{m-1}(\Delta((*, x)))} q_{\alpha}r_{\alpha_m}\cdots r_{\alpha_1} = r_x(1) \sum_{\beta \in  C^{m-2}(\Delta(\Omega_+))} q_{\beta}r_{\beta_{m-1}}\cdots r_{\beta_1} \]
where $\beta = (\beta_1, \ldots, \beta_{m-1})$. We then observe
\begin{align*} 
& \Psi_{(\Gamma_{<x})_+} \left( \sum_{\alpha \in C^{m-1}(\Delta((*, x)))} q_{\alpha}[\alpha_1 \leftarrow \cdots \leftarrow \alpha_{m-1}]_{\alpha_m} \right) \\
 & = \sum _{\alpha \in C^{m-1}(\Delta((*, x)))} q_{\alpha}r_{\alpha_m}\cdots r_{\alpha_1} \\
& = r_x(1) \sum_{\beta \in  C^{m-2}(\Delta(\Omega_+))} q_{\beta}r_{\beta_{m-1}}\cdots r_{\beta_1} \\
& = \Psi_{(\Gamma_{<x})_+} \left( \sum_{\beta \in C^{m-2}(\Delta(\Omega_+))} \left( \bigoplus_{\beta_{m-1} \leftarrow y} q_{\beta}[\beta_1 \leftarrow \cdots \leftarrow \beta_{m-1}]_y \right) \right). 
\end{align*} 
The map $\Psi_{(\Gamma_{<x})_+}$ is an isomorphism and therefore
\begin{align*} & \sum_{\alpha \in C^{m-1}(\Delta((*, x)))} q_{\alpha}[\alpha_1 \leftarrow \cdots \leftarrow \alpha_{m-1}]_{\alpha_m} \\ & = \sum_{\beta \in C^{m-2}(\Delta(\Omega_+))} \left( \bigoplus_{\beta_{m-1} \leftarrow y} q_{\beta}[\beta_1 \leftarrow \cdots \leftarrow \beta_{m-1}]_y \right).
\end{align*}
Thus
\begin{align*} & d_{S^{m-1}(\Gamma_+)} \left( \sum_{\beta \in C^{m-2}(\Delta(\Omega_+))} \left( \bigoplus_{\beta_{m-1} \leftarrow y} [\beta_1 \leftarrow \cdots \leftarrow \beta_{m-1}]_y \right) \right) \\ & = \sum_{\alpha \in C^{m-1}(\Delta((*, x)))} q_{\alpha}[\alpha_1 \leftarrow \cdots \leftarrow \alpha_m]_x.  \end{align*}
And finally,
\[ d_{S^{m-1}(\Gamma_+)} \left( \sum_{\beta \in C^{m-2}(\Delta(\Omega_+))} \left( \bigoplus_{\beta_{m-1} \leftarrow y} [\beta_1 \leftarrow \cdots \leftarrow \beta_{m-1}]_y \right) \right) = [0]_x \]
in $\tilde{H}^{m-1}(\Delta((*, x)))$. We conclude $\Psi_{\Gamma_+} : \tilde{H}^{m-1}(\Delta((*, x))) \rightarrow R_{\Gamma_x}(m,0)$ is an isomorphism, and this completes our proof. 
\end{proof}

\begin{rmk}
Suppose $\Gamma_x$ is Cohen-Macaulay. Using the remark at the end of Section 4 and Corollary \ref{interesting}, we see that the cohomology of $S^{\bullet}(\Gamma_+)$ is the cohomology of $C^{\bullet}(\Delta(\Gamma_+)) = C^{\bullet}(Y)$. 
\end{rmk}

We end this section with several corollaries. We recall the definitions of $\Gamma^{>k}$, $\Gamma_{a, i}$, and $\Gamma(W, k)$ from Section \ref{defprelim}.

\begin{cor}
Let $0 \leq k < m$. Then $\Psi_{(\Gamma^{>k})_+} : S^{\bullet}((\Gamma^{>k})_+) \rightarrow R_{\Gamma}(\bullet, k)$ is a co-chain isomorpshim. 
\end{cor}

The following corollary is Corollary 5.3 from \cite{KS}, without the uniform hypothesis. 

\begin{cor}
Let $\Gamma$ be a finite ranked poset and $*\ne v\in \Gamma$ an element of rank $d+1$.  Then for any $0\le k\le d-1$, 
$$\dim(r_v R_\Gamma(d-1,k)) = \dim(\tilde H^{d-k-1}(\Delta(\Gamma^{>k}_v\setminus\{*,v\}))).$$ 
\end{cor}

The following corollary is (3.1.1) from \cite{RSW4}, without the uniform hypothesis. 

\begin{cor}
Let $\Gamma$ be a finite ranked poset.  Then:
$$  H(R_\Gamma, -t) = 1 + \sum_{i\ge 1}\, \sum_{a\in \Gamma \atop rk_\Gamma(a) \ge i} (-1)^{i-2}\dim\tilde H^{i-2}(\Delta(\Gamma_{a,i})) t^i $$
\end{cor}

\begin{cor}\label{usefulcor}
Let $0 \leq k < n \leq m+1$ and $W \subset \Gamma(n)$. Then $\Psi_{\Gamma(W, k)} : S^{\bullet}(\Gamma(W, k)) \rightarrow R_{\Gamma_S}(\bullet + n - k -1, n-k-1)$ is a co-chain isomorphism. 
\end{cor}


\section{The Main Theorem}\label{weaklyCM}

Assume $\Gamma$ is cyclic. We remind the reader of several facts. By Theorem \ref{equivCM}, $\Gamma$ is Cohen-Macaulay if for all $x \in \Gamma_+$ and all $n < l \leq rk_{\Gamma}(x)$, $\tilde{H}^{n-2}(\Delta(\Gamma_{x, l})) = 0$. Also, from Theorem \ref{KSthm}, $\Gamma$ is Cohen-Macaulay if and only if $\Gamma$ is uniform and $R_{\Gamma}$ is Koszul. Recall again the notation $\Gamma(W, k)$ from Definition \ref{nSnotation} and note that the dimension of $\Delta(\Gamma(W, k))$ is $k$. 

\begin{defn}
Assume $\Gamma$ is cyclic. Then $\Gamma$ is weakly Cohen-Macaulay if for all $x \in \Gamma_+$, $0 \leq k < n \leq rk_{\Gamma}(x)$ and $W \in M_n(\Gamma_x)$,
\[H^{k-1}(S^{\bullet}(\Gamma(W, k))) = 0. \]
\end{defn}


The following fact shows that weakly Cohen-Macaulay is a slightly weaker condition than Cohen-Macaulay.

\begin{prop}
Let $\Gamma$ be a finite ranked cyclic poset. Then $\Gamma$ is Cohen-Macaulay if and only $\Gamma$ is uniform and weakly Cohen-Macaulay. 
\end{prop}

\begin{proof}
In both directions $\Gamma$ is uniform. Thus, if $x \in \Gamma_+$, $\Gamma_x$ is uniform. By Remark \ref{Muniform}, $M_n(\Gamma_x) = \{ \Gamma_x(n) \}$ for all $1 \leq n \leq rk_{\Gamma}(x)$.

For the forward direction, assume $\Gamma$ is Cohen-Macaulay and let $x \in \Gamma_+$ and $0 \leq k < n \leq rk_{\Gamma}(x)$. Set $W = \Gamma_x(n)$. Then, by Corollary \ref{usefulcor}, $H^{k-1}(S^{\bullet}(\Gamma(W, k)))$ is isomorphic to $H^{k-1}(R_{\Gamma_S}(\bullet + n - k -1, n-k-1), d_{\Gamma_{W}})$, which is isomorphic to $H^{k-1}(R_{\Gamma_x}(\bullet + n - k -1, n-k-1), d_{\Gamma_x})$. Since $\Gamma_x$ is Cohen-Macaulay, $R_{\Gamma_x}$ is Koszul. Therefore, the above cohomology is zero, as desired.

For the reverse direction, we will proceed by induction on rank of $\Gamma$ and show that $R_{\Gamma}$ is Koszul. If $\Gamma$ has rank one, there is nothing to prove. We assume $\Gamma$ has rank $m+1$ with $m>0$. Set $\Gamma = \Gamma_z$. By Theorem 2.7 in \cite{KS}, it remains to show $H^n(R_{\Gamma}(\bullet, k)) = 0$ for all $0 \leq k < n \leq m-2$. Let $W = \Gamma(n+2)$. By Corollary \ref{usefulcor}, $H^n(R_{\Gamma}(\bullet, k))$ is isomorphic to $H^{n-k}(S^{\bullet}(\Gamma(W, n-k+1))) = 0$. 
\end{proof}

We now state our main theorem and an important corollary. 

\begin{thm}
Let $\Gamma$ be a finite ranked cyclic poset. Then $\Gamma$ is weakly Cohen-Macaulay if and only if $R_{\Gamma}$ is Koszul. 
\end{thm}

\begin{proof}
We will proceed by induction on the rank, $m+1$, of $\Gamma$. The theorem is true for all $\Gamma$ of rank one. Assume $m > 0$ and that the theorem is true for all cyclic posets of rank less than or equal to $m$. The theorem follows from Theorem \ref{rightannihilator2} and Corollary \ref{usefulcor}. 
\end{proof}

\begin{cor}
Let $\Gamma$ be a finite ranked poset. Then for all $x \in \Gamma$, $\Gamma_x$ is weakly Cohen-Macaulay if and only if $R_{\Gamma}$ is Koszul. 
\end{cor}


\section{A Few Examples and Remarks}\label{examples2}

We first return to Example \ref{ex4.1}.

\begin{example}
The poset $\Theta^*$ from Figure \ref{theta*} is weakly Cohen-Macaulay. By inspection,
$S^0({\Theta^*}_+) = \F [1]_u + \F [1]_t$, $S^1({\Theta^*}_+) = \F[u]_q$, and $S^2({\Theta^*}_+) = S^3({\Theta^*}_+) = 0$. Also, the map $d_{S^0({\Theta^*}_+)}: S^0({\Theta^*}_+) \rightarrow S^1({\Theta^*}_+)$ is surjective. We remark that the dual of a weakly Cohen-Macaulay cyclic poset need not be weakly Cohen-Macaulay.
\end{example}


\begin{example}
The class of weakly Cohen-Macaulay posets includes many non-Cohen-Macaulay lattices. For example the lattice $\Lambda$, in Figure 3, is weakly Cohen-Macaulay but not Cohen-Macaulay. We note that not all lattices are weakly Cohen-Macaulay.

\begin{figure}[h!]
  \centering
  \includegraphics[scale=0.5]{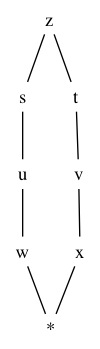}
  \caption{The poset $\Lambda$.} \label{Lambda}
\end{figure} 
\end{example}

It is easy to build weakly Cohen-Macaulay posets from Cohen-Macaulay posets. 

\begin{example}\label{wedge2}
Let $\Gamma$ and $\Omega$ be finite ranked posets with equal rank and unique minimal elements $*_{\Gamma}$ and $*_{\Omega}$. We define
\[ \Gamma \vee \Omega = \Gamma \cup \Omega / (*_{\Gamma} \sim *_{\Omega}). \]
We denote the unique minimal element of $\Gamma \vee \Omega$ by $*$. If $\Gamma$ and $\Omega$ are weakly Cohen-Macaulay, then so is $\overline{\Gamma \vee \Omega}$. In particular, if $\Gamma$ and $\Omega$ are Cohen-Macaulay, then $\overline{\Gamma \vee \Omega}$ is weakly Cohen-Macaulay. Figure \ref{Lambda} above gives an example of this construction. 
\end{example}

We can give a more general construction of weakly Cohen-Macaulay posets: suppose $\Gamma$ is pure and $\Gamma_x$ is Cohen-Macaulay for all $x \in \Gamma$. Then $\overline{\Gamma}$ is weakly Cohen-Macaulay if and only if for all $W \in M_{rk(\overline{\Gamma})-1}(\overline{\Gamma})$, $\overline{\Gamma_W}$ is Cohen-Macaulay.

\bibliographystyle{amsplain}
\bibliography{bibliog}

\def\cprime{$'$}
\providecommand{\bysame}{\leavevmode\hbox to3em{\hrulefill}\thinspace}
\providecommand{\MR}{\relax\ifhmode\unskip\space\fi MR }
\providecommand{\MRhref}[2]{%
  \href{http://www.ams.org/mathscinet-getitem?mr=#1}{#2}
}
\providecommand{\href}[2]{#2}
\begin{thebibliography}{10}

\bibitem{BGSsurvey}
Anders Bj{\"o}rner, Adriano M.~Garsia, and Richard P.~Stanley, \emph{An introduction to
  {C}ohen-{M}acaulay partially ordered sets}, ({B}anff, {A}lta., 1981), NATO Adv. Study Inst. Ser. C: Math. Phys. Sci., 83, Reidel,
  Dordrecht-Boston, Mass., 1982, 583--615.  \MR{0661307}

\bibitem{CPS}
Thomas Cassidy, Christopher Phan, and Brad Shelton, \emph{Noncommutative
  {K}oszul algebras from combinatorial topology}, J. Reine Angew. Math.
  \textbf{646} (2010), 45--63. \MR{2719555}


\bibitem{GGRSW}
Israel Gelfand, Sergei Gelfand, Vladimir Retakh, Shirlei Serconek, and
  Robert~Lee Wilson, \emph{Hilbert series of quadratic algebras associated with
  pseudo-roots of noncommutative polynomials}, J. Algebra \textbf{254} (2002),
  no.~2, 279--299. \MR{1933871}

\bibitem{GRSW}
Israel Gelfand, Vladimir Retakh, Shirlei Serconek, and Robert~Lee Wilson,
  \emph{On a class of algebras associated to directed graphs}, Selecta Math.
  (N.S.) \textbf{11} (2005), no.~2, 281--295. \MR{2183849}
 

\bibitem{KS}
Tyler Kloefkorn and Brad Shelton, \emph{Splitting algebras: {K}oszul, {C}ohen-{M}acaulay and numerically {K}oszul}, J. Algebra, \textbf{422C} (2015),  660--685. \MR{3272095}

\bibitem{Nacin}
David Nacin, \emph{Properties of a minimal non-{K}oszul $A(\Gamma)$}, Noncommutative birational geometry, representations and combinatorics, 215--223, Contemp. Math., \textbf{592}, Amer. Math. Soc., Providence, RI, 2013. \MR{3087946}

\bibitem{Pi}
Dmitri Piontkovski, \emph{Noncommutative {K}oszul filtrations}, http://arxiv.org/pdf/math/0301233v1.pdf.

\bibitem{PP}
Alexander Polishchuk and Leonid Positselski, \emph{Quadratic algebras},
  University Lecture Series, vol.~37, American Mathematical Society,
  Providence, RI, 2005. \MR{MR2177131}
  
\bibitem{Pos}
Leonid Positselski, \emph{Relation between the Hilbert series of dual quadratic algebras does not imply {K}oszulity}, Functional Anal. and its Appl. \textbf{29} (1995), no.~3, 83--87. \MR{1361964}
  
\bibitem{RSW5}
Vladimir Retakh, Shirlei Serconek, and Robert~Lee Wilson, \emph{On a class of {K}oszul algebras associated to directed graphs}, J. Algebra, \textbf{304} (2006), no.~2, 1114--1129. \MR{2265508}

\bibitem{RSW4}
\bysame, \emph{Hilbert series of
  algebras associated to directed graphs and order homology}, J. Pure Appl. Algebra \textbf{216} (2012), no.~6, 1397--1409. \MR{2890509}

\bibitem{RSW1}
\bysame, \emph{Hilbert series
  of algebras associated to directed graphs}, J. Algebra \textbf{312} (2007),
  no.~1, 142--151.  \MR{MR2320451}

\bibitem{RSW3}
\bysame, \emph{Koszulity of splitting algebras associated with cell complexes},
  J. Algebra \textbf{323} (2010), no.~4, 983--999. \MR{2578588}

\bibitem{SadSh}
Hal Sadofsky and Brad Shelton, \emph{The {K}oszul property as a topological
  invariant and measure of singularities}, Pacific J. Math. \textbf{252}
  (2011), no.~2, 473--486.  \MR{2860435}
  
\bibitem{SheltonII}
Brad Shelton, \emph{Splitting algebras II: the cohomology algebra}, http://arxiv.org/abs/1208.2202.pdf.
  
\bibitem{StanleyEC}
Richard~P. Stanley, \emph{Enumerative combinatorics, Volume I}, Second edition, Cambridge Studies in Advanced Mathematics, 49. Cambridge University Press, Cambridge, 2012. \MR{2868112}
  
\bibitem{CW}
Charles A. Weibel, \emph{An introduction to homological algebra}, Cambridge Studies in Advanced Mathematics, vol.~38, Cambridge University Press, Cambridge, 1994. \MR{1269324}

\end{thebibliography}
\end{document}